\documentclass[12pt]{amsart}
\usepackage{amsfonts}
\usepackage{amsmath}
\usepackage{amssymb,latexsym}
\usepackage[mathcal]{eucal}
\usepackage{amscd}

\usepackage[pdftex,bookmarks,colorlinks,breaklinks]{hyperref}
\input xy
\xyoption{all}

\newcommand{\Z}{\mathbb{Z}}

\newcommand{\Ga}{\Gamma}

\newcommand{\del}{\delta}

\newcommand{\ep}{\epsilon}

\newcommand{\ol}{\overline}

\newcommand{\id}{{\rm{id}}}

\swapnumbers
\theoremstyle{plain}
\newtheorem{thm}{Theorem}[section]

\theoremstyle{definition}

\newtheorem{prob}[thm]{Problem}

\begin{document}

\title[On doubly minimal systems and a question regarding product recurrence]
{On doubly minimal systems and a question regarding product recurrence}

\begin{abstract}
We show that a doubly minimal system $X$ has the property that
for every minimal system $Y$ the orbit closure of any pair $(y,x) \in Y \times X$
is either $Y \times X$ or it has the form 
$\Ga_\pi = \{(\pi(x),x) : x \in X\}$ for some factor map $\pi: X \to Y$.
As a corollary we resolve a problem of Haddad and Ott from 2008 regarding product recurrence.
\end{abstract}

\author{Eli Glasner and Benjamin Weiss}

\address{Department of Mathematics\\
     Tel Aviv University\\
         Tel Aviv\\
         Israel}
\email{glasner@math.tau.ac.il}

\address {Institute of Mathematics\\
 Hebrew University of Jerusalem\\
Jerusalem\\
 Israel}
\email{weiss@math.huji.ac.il}




\date{August 3, 2015}

\maketitle

In this note a {\em dynamical system} is a pair $(X,T)$ where $X$ is a compact metric space
and $T$ a self homeomorphism. The reader is referred to \cite{F} for most of the notions 
used below and for the necessary background.

In \cite[Theorem 9.11, p. 181]{F} Furstenberg has shown that a point $x$ of a dynamical system $(X,T)$
is {\em product-recurrent} (i.e. has the property that for every dynamical system
$(Y,S)$ and a recurrent point $y \in Y$, the pair
$(x,y)$ is a recurrent point of the product system $X \times Y$) if and only if
it is a distal point (i.e. a point which is proximal only to itself). 
In \cite{AF}  Auslander and Furstenberg posed the following question: if 
$(x,y)$ is recurrent for all {\bf minimal} points $y$, is $x$ necessarily a distal point? 
Such a point $x$ is called a {\em weakly product recurrent point}.
This question is answered in the negative in \cite{HO}. 

It turns out (see also \cite[Theorem 4.3]{DSY}) that a negative answer was already at hand for 
Harry Furstenberg when he and Joe Auslander posed
this question. In fact, many years before he proved a theorem according to which an F-flow
is disjoint from every minimal system \cite{F}. 
As a direct consequence of this theorem, if $X$ is an F-flow, $x$ a transitive point in $X$, 
$Y$ any minimal system and $y$ any point in $Y$, then the pair $(x,y)$ has a dense orbit
in $X \times Y$. In particular $(x,y)$ is a recurrent point of the product system $X \times Y$.
Thus a transitive point $x$ in an F-flow is weakly product recurrent.
Since such a point  is  never distal, one concludes that
$x$ is indeed weakly product recurrent but not distal.

\vspace{.5cm}

In \cite[Question 5.3]{HO} the authors pose the following natural question:

\begin{prob}\label{prob}
Is every {\bf minimal} weakly product recurrent point a distal point?
\end{prob}
(This was also repeated in \cite[Question 9.2]{DSY}.)

\vspace{.5cm}

In this note we  show that, here again, the answer is negative.
The example is based on a result of \cite{FKS} concerning
POD systems and on the existence of doubly minimal systems
(see \cite{K} and \cite{W}). 
A minimal dynamical system $(X,T)$
is called {\em proximal orbit dense (POD)} if it
is totally minimal and for any distinct points
$u$ and $v$ in $X$, there exists an $0 \not = n \in \Z $ such that
$\Ga_n = \{( T^nx , x) : x \in X\}$ is contained in 
$\ol{\mathcal{O}_{T \times T}(u,v)}$, the orbit closure of $(u,v)$ in 
the product system $X \times X$.

A minimal $(X,T)$ is called {\em doubly minimal} \cite{W}
(or a system having {\em topologically minimal self joinings in the sense of del Junco} \cite{K}) 
if the only orbit closures of $T\times T$ in $X \times X$ are the graphs
$\Ga_m = \{(T^m x,x): x \in X \},\ m \in \Z$  and all of $X \times X$.
Clearly a doubly minimal system is POD.
In \cite{FKS} the authors prove the following striking property of POD systems:

\begin{thm}
If $(Y,S)$ is POD then any minimal $(X,T)$ that is not an extension of $(Y,S)$ is disjoint
from it. 
\end{thm}

For the reader's convenience we reproduce the short proof:

\begin{proof}
Suppose $Y$ is not a factor of $X$ and let $M$ be a minimal subset of $Y \times X$. 
Since $X$ is not an extension of $Y$, there exist $y, y' \in Y$ with $y \not = y'$
and $x \in X$ such that $(y, x), (y', x)  \in M$. 
From the POD property it follows that for some $z \in X$ and $n \not=0$
the points $(y, z)$ and $(T^n y,z)$ are both in $M$. This implies that 
$(T^n \times \id_X) M \cap M \not=\emptyset$ and, as $M$ is minimal,
it follows that $(T^n \times \id_X) M = M$.
Finally, since $Y$ is totally minimal we deduce that $M = Y \times X$,
as required.
\end{proof}

We will strengthen this property for doubly minimal systems as follows:

\begin{thm}\label{DM}
If $(Y,S)$ is doubly minimal and $(X,T)$ is any
minimal system then the orbit closure of any point $(y,x) \in Y \times X$
is either all of $Y \times X$ or it is the graph 
$\Ga_\pi = \{(\pi(x),x) : x \in X\}$
of some factor map $\pi: X \to Y$.
\end{thm}

\begin{proof}
Let $Y$ be a doubly minimal system.
In particular $Y$ is weakly mixing and has the POD property.
Let $X$ be a minimal system.
By \cite{FKS} either $X$ and $Y$ are disjoint, or $Y$ is a factor of $X$. 
In the first case the product system $Y \times X$ is minimal.

So we now assume that there is a factor map $\pi : X \to Y$.
We consider an arbitrary point $(y_0, x_1) \in Y \times X$ and denote $y_1 = \pi(x_1)$.
We will denote the acting transformation on both $X$ and $Y$ by the letter $T$.

\vspace{.5cm}

{\bf Case 1:} $y_1 = T^n y_0$ for some $n \in \Z$.

In this case the orbit closure $\overline{\mathcal{O}_{T \times T}(y_0, x_1)}$
has the form 
$$
\Ga_{\pi \circ T^{-n}} = \{(\pi(x), T^n x) : x \in X\},
$$ 
and is isomorphic to $X$.

\vspace{.5cm}

{\bf Case 2:} $y_1 \not \in \mathcal{O}(y_0)$.

Recall that by double minimality we have in this case that 
\begin{equation*}
\overline{\mathcal{O}_{T \times T}(y_0, y_1)} = Y \times Y.
\end{equation*}
Also note that, as the union of the graphs $\bigcup_{n \in \Z} \Ga_n$,
where $\Ga_n = \{(T^n x, x) : x \in X\}$, is dense in $X \times X$, the 
union of the graphs $\bigcup \Ga_{\pi \circ T^{n}}$ is dense in $Y \times X$.

\vspace{.5cm}

Let $(u,v)$ be an arbitrary point in $Y \times X$ and fix an $\ep >0$.

(i) Choose a point $w \in X$ and $m \in \Z$ such that $(\pi(w), T^m w) \overset{\ep}{\sim} (u,v)$. 

(ii) Choose
a sequence $n_i \in \Z$ such that for some point $z \in X$ 
\begin{equation*}
T^{n_i}(y_0, x_1) \to (\pi(w), z),
\qquad {\text{with}} \qquad  \pi(z) = T^m \pi(w).
\end{equation*}

(iii) Choose a sequence $k_j \in \Z$ such that 
\begin{gather*}
T^{k_j}z  \to T^m w, \qquad 
{\text{whence}} \qquad 
T^{k_j} \pi(z)  = T^{k_j} T^m \pi(w)  \to T^m \pi(w),\\
 \text{and} \qquad  T^{k_j} \pi(w)  \to  \pi(w).
\end{gather*}

\vspace{.5cm}

Now
$$
\lim_j \lim_i T^{k_j} T^{n_i} (y_0, x_1) 
= \lim_j T^{k_j} (\pi(w), z) = (\pi(w), T^m w)  \overset{\ep}{\sim} (u, v).
$$
Since $\ep >0$ is arbitrary we conclude that$(u, v) \in \overline{\mathcal{O}_{T\times T}(y_0, x_1)}$,
hence $\overline{\mathcal{O}_{T\times T}(y_0, x_1)} = Y \times X$.
 \end{proof}

As a corollary of this theorem and the fact that 
there are weakly mixing doubly minimal
systems (\cite{K} and \cite{W}) we get a negative answer to Problem \ref{prob}.

First note that a minimal weakly mixing system does not admit a distal point.
One way to see this is via the fact that in a minimal weakly mixing
system $X$, for every point $x \in X$ there is a dense $G_\del$ subset $X_0 \subset X$
such that for every $x' \in X_0$ the pair $(x,x')$ is proximal; see
\cite[Theorem 9.12]{F}, or \cite{AK} for an even stronger statement.

\begin{thm}
There exists a minimal dynamical system $Y$ which is weakly mixing
(hence in particular does not have distal points)
yet it has the property that for every minimal system $X$
every pair $(y,x) \in Y \times X$ is recurrent.
\end{thm}

 \begin{proof}
Let $Y$ be a weakly mixing doubly minimal system and $X$ a minimal system.
By \cite{FKS} either $X$ and $Y$ are disjoint, or $Y$ is a factor of $X$. 
In the first case the product system $Y \times X$ is minimal and, in particular, 
every pair $(y,x)$ is recurrent.

In the second case we have, by Theorem \ref{DM},
$$
\overline{\mathcal{O}_{T\times T}(y, x)} = \Ga_\pi = \{(\pi(z),z) : z \in X\},
$$
for a factor map $\pi : X \to Y$. 
Again $(y,x) = (\pi(x),x)$ is recurrent and the proof is complete.
 \end{proof}

\end{document}